\documentclass[letterpaper]{amsart}
\usepackage{amsfonts,amssymb,amscd,amsmath,enumerate,url,verbatim}
 \usepackage[dvips]{epsfig}
 \usepackage{amsgen, amstext,amsbsy,amsopn, amsthm, amsfonts,amssymb,amscd,amsmat
 h,euscript,enumerate,url,verbatim,calc,xypic}

 \usepackage{latexsym}
 \usepackage{graphics}
 \usepackage{color}

 \newcommand{\m}{\mathfrak{m} }

\newcommand{\indeg}{\operatorname{indeg}}

\theoremstyle{plain}
 \newtheorem{question}{Question}

\newtheorem{thm}{Theorem}[section]
\newtheorem{lem}[thm]{Lemma}
\newtheorem{cor}[thm]{Corollary}
\newtheorem{prop}[thm]{Proposition}
\newtheorem{rem}[thm]{Remark}
\newtheorem{defin}[thm]{Definition}

\newtheorem{defen}[thm]{Definition}
\newtheorem{exam}[thm]{Example}
\newtheorem{defin-rem}[thm]{Definition and Remark}

\def\FF{\mathbf  F}

\def\GG{\mathbf  G}

\def\NN{ {\bf N} }

\newcommand{\D}{\operatorname{D}}

\newcommand{\dep}{\operatorname{depth}}

\newcommand{\Tor}{\operatorname{Tor}}
\newcommand{\Ker}{\operatorname{ker}}

\newcommand{\reg}{\operatorname{reg}}

\newcommand{\lin}{\operatorname{lin}}
\newcommand{\ld}{\operatorname{ld}}
\newcommand{\HH}{\operatorname{H}}
\begin{document}

\title[ ] { Regularity and linearity defect of  modules over local rings}

\begin{abstract}
Given a finitely generated module $M$ over a commutative local ring (or a standard graded $k$-algebra) $(R,\m,k) $  we  detect its complexity in terms of numerical invariants coming from suitable
$\m$-stable filtrations $\mathbb{M}$ on $M$. We study the Castelnuovo-Mumford regularity of $gr_{\mathbb{M}}(M) $ and the linearity defect of $M,$ denoted $\ld_R(M),$ through a deep investigation based on the theory of standard bases. If $M$ is a graded $R$-module, then $\reg_R(gr_{\mathbb{M}}(M)) <\infty$ implies $\reg_R(M)<\infty$ and the converse holds provided $M$ is of homogenous type. An analogous result can be proved in the local case in terms of the linearity defect. Motivated by a positive answer in the graded case, we present for  local rings a partial answer to a question raised by Herzog and Iyengar of whether $\ld_R(k)<\infty$
implies $R$ is Koszul.
 \end{abstract}


 \author[R.~ Ahangari Maleki]{Rasoul Ahangari Maleki}
\address{Rasoul Ahangari Maleki\\ Faculty of Mathematical Sciences and Computer, Kharazmi  University, Tehran, Iran}
\email{std\_ahangari@khu.ac.ir}

\author[M.~E.~Rossi]{Maria Evelina Rossi}
\address{Maria Evelina Rossi\\ Department of Mathematics, University of Genova \\ Via Dodecaneso 35, 16146-Genova, Italy}
\email{rossim@dima.unige.it}





\thanks{The second  author  was  supported  by MIUR, PRIN 2010-11 (GVA).  This work was partly accomplished while the first author  was visiting the University of Genoa. }

\maketitle

\section*{Introduction and notation}
Throughout this paper $(R,\m,k)$ is a commutative Noetherian local ring (or a standard graded $k$-algebra)
with   maximal ideal (or  homogeneous maximal ideal)   $\m$ and residue field $k. $ All the modules we consider are
finitely generated over $R$.
Let $M$ be an $R$-module,   according to   \cite{RV-Lect}, we say that a descending filtration of
submodules $\mathbb{M}=\{\mathfrak{F}_{p}M\}_{p\geq 0}$
of  $M=\mathfrak{F}_0 M$ is an $\m$-filtration if $\m\ \mathfrak{F}_{p}M\subseteq \mathfrak{F}_{p+1}M$
for every $p\geq 0$, and  $\mathbb{M} $ is  an $\m$-stable (or good) filtration  if $\m ~\mathfrak{F}_{p}M=\mathfrak{F}_{p+1}M$
for all sufficiently large $p$. In the following a filtered module $M$
will be always an $R$-module equipped with an  $\m$-stable  filtration $\mathbb{M}$.  Define
 $$gr_{\mathbb{M}}(M)=\bigoplus_{p\geq 0}(\mathfrak{F}_{p}M/\mathfrak{F}_{p+1}M)$$
the {\bf associated graded module} to $M$ with respect to  the filtration $\mathbb{M}$.  If $\mathbb{M}= \{ \m^pM\} $ is the $\m$-adic filtration on $M, $ we denote $gr_{\mathbb{M}}(M)$ simply by $M^g.$ Then $gr_{\mathbb{M}}(M) $ is a graded $R^g$-module.
  When $R$ is a standard graded $k$-algebra,  then $R^{g}$ is naturally isomorphic to $R$.

The main goal of this paper is to study  properties  of the  module M, such as the linearity defect, the Koszulness  and the regularity  by means of the graded structure  of $gr_{\mathbb{M}}(M) $ for a given     $\m$-stable  filtration $\mathbb{M}.$  The  strength  of our approach comes from a deep investigation of   the interplay between a minimal free resolution of $gr_{\mathbb{M}}(M)$ as $R^g$-module and a free resolution of $M$ as $R$-module coming from  the theory of the standard bases.    \\

  \indent
To avoid triviality, we assume that $gr_{\mathbb{M}}(M)$ is not zero or
equivalently $M\neq 0$. We   consider a minimal free resolution of $M$ as finitely generated  $R$-module, that is, a complex of free $R$-modules
$$\FF: \cdots \to  F_{i+1} \stackrel{\phi_{i+1}} \longrightarrow  F_{i}  \stackrel{\phi_{i}} \longrightarrow  F_{i-1} \to \cdots \to F_1\stackrel{\phi_{1}} \longrightarrow F_0\to 0$$
such that $\HH_i(\FF)=0$ for $i>0$ and $\HH_0(\FF)=M$, $  \phi_{i+1}(F_{i+1})  \subseteq \m  F_i$ for every $i$. Such a resolution exists and it is unique up to an isomorphism of complexes.
 By definition, the $i$-th Betti number $\beta_i^R(M)$ of $M$   is the rank of $F_i, $ that is $\beta_{i}^R(M)=\dim_k  \Tor^R_i(M,k). $   It is well known that,  given a  minimal graded free resolution $\mathbf{G}$ of
$gr_{\mathbb{M}}(M)$ as $R^g$-module, one  can build up  a free resolution $\FF$ of $M$ (not necessarily minimal)
equipped with a special filtration $\mathbb{F}$ on $\FF$
such that $gr_{\mathbb{F}}(\mathbf{F})=\mathbf{G}$ (see   \cite[2.4]{S}). In some cases   the  process for obtaining  a  {\it{minimal}}  free resolution of $M$ starting from $ gr_{\mathbb{M}}(M)$ is under control via special cancellations (see  \cite[3.1]{RS}).
As a consequence   of this construction  $$\beta_{i}^R(M) \le \beta_{i}^{R^g}(gr_{\mathbb{M}}(M)) $$ and $M$ is said of {\it{homogeneous type}} with respect to $\mathbb{M}$ if the equality holds for every $i \ge 0. $ This is  equivalent to say that the   resolution $\mathbb{F}$ in the above construction is minimal.
Interesting examples of modules  of homogeneous type are given in literature, see for example \cite{HRV} or \cite[1.10, 1.12]{RS}. It is well known that if $gr_{\mathbb{M}}(M) $ has a linear resolution, then $M$ is of homogeneous type.

 If $R$ is a   standard graded $k$-algebra and $M$ is a finitely generated graded $R$-module, then  $ \Tor^R_i(M,k)$ inherits the graded structure  and the $(i,j)$-th graded Betti number $\beta_{ij}^R(M)= \dim_k  \Tor^R_i(M,k)_j $  of $M$ is the number of copies of $R(-j)$   that appear in $F_i$  (recall  $R(-j)_i := R_{i-j}$).
   We set
$$t_i^R(M)=\sup\{ j : \beta_{ij}^R(M)\neq 0\}$$
where, by convention,  $t_i^R(M)=-\infty$ if $F_i=0$. By definition, $t_0^R(M)$ is the largest degree of a minimal generator of $M$.
An  important invariant  associated to a  minimal free resolution of $M$ as   $R$-module is  the Castelnuovo-Mumford regularity
$$\reg_R(M)=\sup\{ j-i : \beta_{ij}^R(M)\neq 0\}=\sup\{ t_i^R(M)-i : i\in \NN \}.$$
 \vskip 2mm
 It is clear that $\reg_R(M) $ can be infinite.  In the graded case we have a double opportunity: we can take advantage both of the graded $R$-structure of $M$ and of the graded $R$-structure of $gr_{\mathbb{M}}(M). $ We present explicit bounds on the graded Betti numbers of $M$ in terms of the degrees of a minimal system of generators of $M$ and of the Betti numbers of  $gr_{\mathbb{M}}(M). $ In particular we prove that if  there exists  a filtration $\mathbb{M}$ of  $M$ such that
   $\reg_{R}(gr_{\mathbb{M}}(M))< \infty, $ then $\reg_{R}(M)< \infty. $
If  $M$ is of homogeneous type with respect to $\mathbb{M}, $ then also the
converse holds, see Corollary \ref{corol}.  Example \ref{contro} shows that the assumption is necessary.

If $R$ is  graded,    the residue field $k$  has a  special behaviour.    Avramov and Peeva in \cite{AP} proved that  either $k$ admits a linear resolution, that is  $ \reg_R(k)=0, $  or  $\reg_R(k) $ is  infinite.  Following the classical definition given by Priddy, if $k$ admits a linear resolution, we say that $R$ is Koszul.


For local rings  we say that $R$ is Koszul if and only if  $R^g $ is Koszul.  One can reformulate the analogous of the statement proved by Avramov and Peeva for local rings  replacing  the concept of regularity with a new invariant that can be defined  for any finitely generated module $M$ over a local ring.
 This  invariant is the
linearity defect,  denoted  $\ld_R(M).$  This notion
was introduced by Herzog and Iyengar in \cite{HI} and studied further by Iyengar and
Roemer in \cite{IR},  see also  Section 3.  The linearity defect measures how far $M^g$ is  from having a linear resolution.   Notice that if  $\ld_R(M)< \infty,  $ then the Poincar\'{e} series   $P_M(t) =\sum_{i \ge 0} \beta_i(M) t^i $ of $M$ is rational sharing all modules a common  denominator depending only on $R^g$, see \cite[1.8]{HI}.
 \vskip 2mm
 In \cite[1.14]{HI} the following challenging question has been stated.
\begin{question}
\label{loc1} If  $\ld_R(k)< \infty $  does it follow  that  $\ld_R(k)=0$?  \end{question}
In the graded case $\ld_R(k)< \infty $ implies $\reg_R(k) < \infty, $ then $R$ is Koszul   as before mentioned by the result of Avramov and Peeva.  In the local case the problem is still open and it is one of the aims  of our investigation. Interesting answers are given in a recent paper of \c{S}ega  for certain classes of local rings, see \cite{S-lin}.
 \vskip 2mm
The above question  is equivalent  to prove that $\ld_R(k)< \infty $ implies $\ld_{R^g}(k)< \infty. $    In analogy, since $k=k^g, $  a  more general question arises:
\begin{question}
\label{loc2}
Let $M$ be a finitely generated $R$-module.  If   $\ld_R(M)< \infty $  does it follow that  $\ld_{R^g}(M^g)< \infty $? \end{question}

 In general the  answer to Question \ref{loc2} is negative,  nevertheless this does not disprove  Question \ref{loc1}.  For instance,    Example \ref{contro}   presents  a  graded module of finite projective dimension, hence of finite linearity defect,  but  the regularity of  $M^g$ is infinite,  hence  the linearity defect of $M^g$ is infinite by  \cite[1.2]{HI}. In view of giving  an answer to Question \ref{loc1}, since $k$ is a cyclic $R$-module,  could be  interesting to analyze also Example \ref{cyclic} which provides a cyclic module  with same features.

\vskip 2mm

  In Proposition \ref{ld}  we prove that if  $\ld_{R}(M)=d<\infty, $ then $$\reg_{R}(M)=\max\{t_{i}(M)-i: 0\leq i\leq d\}  $$   as  expected specification of   \cite[1.9]{HI} where it is proved that if $\ld_{R}(M) <\infty, $ then  $\reg_{R}(M) <\infty. $
In particular, if $M$ is a Koszul $R$-module,  then $\reg_{R}(M)=t_{0}(M).$  Theorem  \ref{ld- local} presents the analogous result in the local case.  We   prove that if $M$ is a module of homogeneous type with respect to the  filtration $\mathbb{M},  $ then $\ld_{R}(M)=d<\infty  $ implies   $\reg_{R^g}(gr_{\mathbb{M}}(M)) =\max\{t_{i}(gr_{\mathbb{M}}(M))-i: 0\leq i\leq d\}.$ This allow us to present   a positive answer to Question 1 when $k$ is of homogeneous type with respect to a stable filtration $\mathbb{M},  $ see Proposition \ref{k}.

\smallskip
\section{Preliminaries on filtered modules}
We refer to \cite{RS} and \cite{RV-Lect}  for more facts concerning filtered modules and filtered complexes. We include in this section  the results and we fix the definitions  that will be used in the present paper.
\vskip 2mm
   If N is a submodule of an $R$-module $M$, by Artin-Rees lemma, the collection
$\{N\cap \mathfrak{F}_{p}M\mid p\geq 0\}$ is an  $\m$-filtration of $N$ and we denote  $gr_{\mathbb{M}}(N)$   the associated graded module induced by $\mathbb{M}  $ on $N.$ Since
\begin{equation}\label{g-sub}
(N\cap \mathfrak{F}_{p}M)/(N\cap \mathfrak{F}_{p+1}M)\cong (N\cap \mathfrak{F}_{p}M+\mathfrak{F}_{p+1}M)/\mathfrak{F}_{p+1}M,
\end{equation}
$gr_{\mathbb{M}}(N)$ is a graded submodule of $gr_{\mathbb{M}}(M). $
If $x\in M\setminus {0}$, we denote by $v_{\mathbb{M}}(x) $ the largest integer $p$
such that $x\in  \mathfrak{F}_{p}M $  (the so-called valuation of $x$ with respect to $\mathbb{M}$) and we
denote by $x^{*}$ or $gr_{\mathbb{M}}(x)$ the residue class of $x$ in $\mathfrak{F}_{p}M/\mathfrak{F}_{p+1}M$
where $p=v_{\mathbb{M}}(x)$. If $x=0,$ we set $v_{\mathbb{M}}(x)=+\infty$.\\
\indent Using (\ref{g-sub}), it is clear that $$gr_{\mathbb{M}}(N)=\langle gr_{\mathbb{M}}(x): x\in N\rangle.$$
\begin{defen}
Let $M$ be a  filtered $R$-module. A subset $S=\{f_{1},\cdots ,f_{s}\}$ of
$M$ is called a \textbf{standard basis} of $M$ with respect to ${\mathbb{M}} $ if
$$gr_{\mathbb{M}}(M)=\langle gr_{\mathbb{M}}(f_{1}),\cdots,gr_{\mathbb{M}}(f_{s})\rangle.$$

\noindent If any proper subset of $S$ is not a standard basis of $M$,
we call $S$ a \textbf{minimal standard basis}.
\end{defen}

We remark that the number of generators and the corresponding valuations of a minimal standard basis depend only on $M$ and
${\mathbb{M}}. $
\begin{defin}
Let $M$ and $N$ be filtered $R$-modules and $f:M\rightarrow N$ be
a  homomorphism of $R$-modules. Then  $f$ is said to be a homomorphism of filtered
modules if $f(\mathfrak{F}_{p}M)\subseteq \mathfrak{F}_{p}N$ for every $p\geq 0$ and $f$
is said to be \textbf{strict} if $f(\mathfrak{F}_{p}M)=f(M)\cap \mathfrak{F}_{p}N$ for every $ p\geq 0$.
\end{defin}

 Let $F=\bigoplus_{i=1}^{s}Re_{i}$ be a free $R$-module of
 rank $s$ and $v_{1},\cdots,v_{s}$ be integers. We define
 the filtration $\mathbb{F}=\{\mathfrak{F}_{p} F: p\in\mathbf{Z}\}$ on $F$ as follows
\begin{equation}\label{special}   \mathfrak{F}_{p} F :=\bigoplus_{i=1}^{s}\m^{p-v_{i}}e_{i}=\{(a_{1},\cdots,a_{s}):a_{i}\in \m^{p-v_{i}}\}. \end{equation}
 If $i\leq 0$ we set $\m^{i}=R.$
 We denote the filtered free $R$-module $F$ by $\bigoplus_{i=1}^{s}Rv_{i}$
 and we call it special filtration  on $F$. It is clear that the filtration $\mathbb{F}$
 defined above
 is an $\m$-stable filtration.

 If  $({\bf{F}}, d.) $ is a complex of finitely generated free $R$-modules ($d_i : F_i \to F_{i-1} $ denotes  the $i$-th differential map),  {\it{ a special filtration}}  on  {\bf{F}}  is a special filtration on each $F_i $ that makes $({\bf{F}}, d.) $ a filtered  complex.

\vskip 2mm

\indent Let $M$ be a $R$-module equipped with the filtration ${\mathbb{M}}  $   and let   $S=\{f_{1},\cdots,f_{s}\}$
be a system of elements of $M$ with  $v_{\mathbb{M}}(f_{i})$   the
corresponding valuations. Let $F=\bigoplus_{i=1}^{s}Re_{i}$ be a free $R$-module of rank $s$
equipped with the filtration $\mathbb{F}$ where
$v_{i}=v_{\mathbb{M}}(f_{i})$. Then we denote the filtered free
$R$-module $F$ by $\bigoplus_{i=1}^{s}Rv_{\mathbb{M}}(f_{i})$
and hence $v_{\mathbb{\mathbb{F}}}(e_{i})=v_{\mathbb{M}}(f_{i}).$\\

Let $d  :F\rightarrow M$ be a morphism of filtered $R$-modules defined by
\begin{equation} \label{d}  d(e_{j})=f_{j}. \end{equation}
It is clear that $d  $ is morphism of filtered $R$-modules and
$gr_{\mathbb{F}}(F)$ is isomorphic to the graded free $R^{g}$-module
$\bigoplus_{i=1}^{s}R^{g}(-v_{\mathbb{M}}(f_{i}))$ with a basis
$(e_{1},\cdots ,e_{s})$ where $\deg(e_{i})=v_{\mathbb{M}}(f_{i}).$
In particular $d $ induces a natural graded morphism (of degree zero)
$gr(d):gr_{\mathbb{F}}(F)\rightarrow gr_{\mathbb{M}}(M)$
sending $e_{j}$ to $gr_{\mathbb{M}}(f_{j}).$\\

In particular   $ gr ( \dot ) $ is a
functor from the category of the filtered $R$-modules into the
category of the graded $R^g$-modules. \\
 \indent Let $c=^{t}(c_{1},\cdots,c_{s})$ be an element of $F$ and $^t( \ \ )$ denotes the transposed vector.
By the definition of the filtration $\mathbb{F}$
on $F $, we have
\[v_{\mathbb{F}}(c)=\min\{v_{R}(c_{i})+v_{\mathbb{M}}(f_{i})\}\leq v_{\mathbb{M}}(d(c))\]
 Set
 $gr_{\mathbb{F}}(c)=^{t}(c^{'}_{1},\cdots,c^{'}_{s})$ and $v=v_{\mathbb{F}}(c)$, then

\begin{equation}\label{gr}
c^{'}_{i} =
\begin{cases} gr_{\m}(c_{i}) &\text{if}\ v_{R}(c_{i})+v_{\mathbb{M}}(f_{i})=v\\
0&\text{if}\ v_{R}(c_{i})+v_{\mathbb{M}}(f_{i})>v
\end{cases}
\end{equation}\\

We have a canonical embedding $gr_{\mathbb{F}}(\Ker(d ) ) \rightarrow \Ker(gr(d)).$
 If $gr(d)$ is surjective,
then $ d $ is a strict surjective homomorphism, equivalently  $$gr_{\mathbb{F}}(\Ker(d) ) = \Ker(gr(d)).$$
With the previous notation,  extending a result  of   Robbiano and Valla in \cite{RV}, Shibuta   characterized
 the  standard bases of a filtered module  (see \cite{Shibuta}).

\begin{thm}\label{stand-base}
Let $M$ be a  filtered $R$-module and  $ \{f_{1},\cdots ,f_{s}\}$ elements  of
$M$. The following facts are equivalent:
\begin{enumerate}
\item $\{f_{1},\cdots ,f_{s}\}$ is a standard basis of $M$ with respect to $\mathbb{M}.$
\item $\{f_{1},\cdots ,f_{s}\}$ generates $M$ and $ d $ is strict.
\item $\{f_{1},\cdots ,f_{s}\}$ generates $M$ and $gr_{\mathbb{F}}(\Ker(d ) ) = \Ker(gr(d )).$
\end{enumerate}
\end{thm}





\smallskip
\section{Regularity of  filtered graded modules}
 Let  $R$ be  a standard graded algebra over a field $k$ with homogeneous maximal ideal $\m$ and let $M$ be  a filtered  graded  $R$-module, that is a graded module $M= \oplus_{i\ge0}  M_i$ equipped with a $\m$-stable  filtration  ${\mathbb{M}}.$ In this  section we will present  a  comparison between the regularity of  $M$ and the regularity of  $ gr_{\mathbb{M}}(M). $ The approach is elementary and our investigation is  mainly based on the theory of standard bases.
  \vskip 2mm
In the  graded module $M $  we say that an element $x$ is homogeneous if  $x\in M_i   $ for some $i.$
\begin{defin}
Let $M$ be a filtered  \textbf{graded}   $R$-module and let $S=\{f_{1},\cdots ,f_{s}\}$
be a standard basis of $M$. $S$ is   a \textbf{homogeneous standard basis}
of $M$ if $f_{i}$ is a homogeneous element of $M$
for every $1\leq i\leq s$.

\noindent The filtration  $\mathbb{M}=\{\mathfrak{F}_{p}M\}_{p\ge 0} $ is   a \textbf{graded filtration} if   $ \mathfrak{F}_{p}M$ is a graded submodule  of $M $ for each $p\geq 0. $  \end{defin}

  For example the $\m$-adic filtration
$\{\m^{p}M\}_{p\geq 0}$ is a graded filtration.
When $\mathbb{M}$ is a graded filtration on $M$, then
 $M$ admits  a homogeneous standard basis.
In general a filtered graded module
has not necessarily  a homogeneous standard basis. For instance, let $k$ be a field, set $R=k[X,Y]/(X^{3},Y^{4})$, we write $x$ and $y$ for the residue class of $X$ and $Y$ in $R$. Set  $\m=(x,y)$ and $I=\mathfrak{F}_{0} I=(x^{2},y^{3})$, $\mathfrak{F}_{1} I =(y^{3}+x^{2})$, $\mathfrak{F}_{p} I=\m F_{p-1}I $ for $p\geq 2$ . Note that $\m \mathfrak{F}_{0} I=(xy^{3},x^{2}y)$ and  we have $\m \mathfrak{F}_{0}I \subset \mathfrak{F}_{1} I$. So $\mathbb{I}=\{\mathfrak{F}_{p} I\}_{p\geq 0}$ is an $\m$-stable filtration on $I$. In this case  $I$ has no a  homogeneous standard basis with respect to $\mathbb{I}. $

 \vskip 2mm

From now on if $M$ is a graded module, a filtration on $M$ will be always a  graded filtration.
\vskip 2mm
Let $N$ be a
 graded $R$-module equipped with the filtration $\mathbb{N}=\{\mathfrak{F}_pN\}_{p\ge 0}.$   For every non-zero homogeneous element $x \in N$ we have  two integers attached to $x.$    We say that $x$ has degree $i$ and we write $\deg(x)=i $  if $x \in N_i$   and we say that $x$ has valuation $p=v_{\mathbb{N}}(x) $ if $ x \in \mathfrak{F}_pN \setminus \mathfrak{F}_{p+1}N.$

 If  $\{x_{1},\cdots,x_{n}\}$ is a
 minimal homogeneous generating set of $N, $
 we denote by $\D(N)$  the set
 $\{\deg(x_{i}): 0\leq i\leq n\}$. This set is uniquely determined
 by $N$. If $N=(0),$ we set $\D(N)=\varnothing.$
Let $\{f_{1},\cdots,f_{s}\}$ be a minimal homogeneous standard basis of $N$
with respect to  $\mathbb{N}$. We set\\
\[\Delta_{\mathbb{N}}(N)=\{\deg(f_{j})-v_{\mathbb{N}}(f_{j}): 1\leq j \leq s\}\]
and
\[v_{\mathbb{N}}(N)=\max \Delta_{\mathbb{N}}(N)\]
\[u_{\mathbb{N}}(N)=\min \Delta_{\mathbb{N}}(N)\]\\
When  $\mathbb{N}$ is the $\m$-adic filtration
 then $\{f_{1},\cdots,f_{s}\}$ is also a minimal homogeneous generating
 set of $N$ and we have $v_{\mathbb{N}}(f_{j})=0$. Hence one has
\[\Delta_{\mathbb{N}}(N)=\D(N)\] and
\[\ v_{\mathbb{N}}(N)=t_{0}(N),\ \ u_{\mathbb{N}}(N)=\indeg(N).\]

\vskip 2mm
The following result completes  a well known comparison between the numerical invariants of a graded free resolution of $ gr_{\mathbb{M}}(M)  $  as $R^g$-module and a free resolution of $M$ as $R$-module. The local version of the first part of the proposition   holds for any commutative ring $R $ (see \cite[2.4]{S}), we insert here a refinement  in the case of graded modules.

\begin{prop}\label{free-res}
Let  $R$ be  a standard graded $k$-algebra and let $M$ be  a filtered  graded  $R$-module equipped with an  $\m$-stable graded  filtration  ${\mathbb{M}}.$   Then there exists a \textbf{graded}  free resolution
 $\mathbf{F}  $ of $M$ with a special filtration $\mathbb{F}$ on it such that
$\mathbf{G}:=gr_{\mathbb{F}}(\mathbf{F})$ is a minimal
graded free resolution of $gr_{\mathbb{M}}(M) $ as $R$-module.
\vskip 2mm
\noindent Denote by $d_i : F_i \to F_{i-1} $ (resp. $\delta_i : G_i \to G_{i-1}$)  the differential maps  of $\mathbf{F}  $ (resp. $\mathbf{G}$) and by $\mathbb{F}_i $ the special filtration on $F_i.  $  Then for  all $i \ge 0$ we have:

\begin{enumerate}
\item [$(i)$] $gr_{\mathbb{F}_{i}}(d_{i})=\delta_{i} $
\item [$(ii)$] $\Ker(d_{i})$ admits a homogeneous standard basis with respect to $\mathbb{F}_i  $
\item[$(iii)$] $\Ker({\delta_{i}})=gr_{\mathbb{F}_{i}}(\Ker(d_{i})) $
\item[$(iv)$] $\Delta_{\mathbb{F}_{i}}(\Ker(d_{i})) \subseteq \Delta_{\mathbb{M}}(M)$ .
\end{enumerate}

\end{prop}
\begin{proof}
\indent  Let $f_{1},\cdots,f_{\beta_{0}}$ be a minimal
homogeneous standard basis of $M$ with respect to the filtration
$\mathbb{M}$. Set  $a^{'}_{0 j}:=\deg(f_{j})$ and
$a_{0 j}:= v_{\mathbb{M}}(f_{j})$.   Then $\Delta_{\mathbb{M}}(M)=\{ a^{'}_{0 j}-a_{0 j}: j=1,\dots, \beta_{0}\}$.

\noindent Define  $F_{0}=\bigoplus_{j=1}^{ \beta_{0}}R(-a^{'}_{0 j})$
the graded $R$-free module of rank $\beta_{0}$ equipped with the special filtration $\mathbb{F}_0$ with
respect to  the integers $v_{\mathbb{M}}(f_{1}),\cdots,v_{\mathbb{M}}(f_{\beta_{0}})$ (as  defined in (\ref{special})),
hence $v_{\mathbb{F}}(e_{0j})=v_{\mathbb{M}}(f_{j}). $  We define
\[d_{0}:F_{0}\rightarrow M\]
to be the homogeneous homomorphism
such that $d_{0}(e_{0 j})=f_{j}$. We remark that $R^g=R $ and we  define
\[\delta_{0}:G_{0}=\bigoplus_{j=1}^{\beta_{0}}R(-a_{0 j})\rightarrow gr_{\mathbb{M}}(M)\]
with $\delta_{0}(e_{0 j})=gr_{\mathbb{M}}(f_{j})$.
We have  the following diagram:
\begin{equation}
\xymatrix{
0 \ar[r] &L=\Ker(d_0)  \ar[r]&  \bigoplus_{j=1}^{\beta_0}R(-a^{'}_{0 j}) \ar[d]^{gr_{\mathbb{F}}}\ar[r]^-{d_0} &M\ar[d]^{gr_{\mathbb{M}}}  \\
0 \ar[r] &K=\Ker (gr(d_0))  \ar[r]&  \bigoplus_{j=1}^{\beta_0}R(-v_{\mathbb{M}}(f_{i}))
 \ar[r]^-{\delta_0=gr(d_0)} &gr_{\mathbb{M}}(M)}
\end{equation}

Since $\{f_{1},\cdots,f_{\beta_{0}}\}$ is a standard basis of $M$,
it generates $M$. Hence  $d_{0}$ is surjective
and strict.  In particular,  by Theorem \ref{stand-base},     $gr_{\mathbb{F}_{0}}(\Ker(d_{0}))= \Ker(\delta_{0})$.

The special filtration $\mathbb{F}_0$
is a graded filtration on the graded free $R$-module
$F_0$. Since $d_0$ is homogeneous,  then $L=\Ker(d_0) $ is a
graded submodule of $F_0 $ and
the induced filtration on $L, $ that is ${\mathbb{F}_{0}}\cap L $ is a graded filtration.
Therefore $L=\Ker(d_0) $ admits  a homogeneous standard basis with respect to the
induced filtration. Note that
$gr_{\mathbb{F}_{0}}(e_{0 j})=e_{0 j}$ and   we may write\\
\[\deg(f_{j})= a^{'}_{0 j}=a_{0 j}+c   \]\\
with  $c=\deg(f_{j})-v_{\mathbb{M}}(f_{j}) \in\Delta_{\mathbb{M}}(M).$ Hence  $(i)$, $(ii)$, $(iii)$ are satisfied for $i=0.$
To prove $(iv)$, let $\{h_{1},\cdots,h_{\beta_{1}}\}$ be a minimal homogeneous standard basis of $\Ker(d_0)$. Assume that
$gr_{{\mathbb{F}}_0}(h_i)=(x_{i 1},\cdots,x_{i\beta_0})$. There exists $j$ such that $x_{ij}\neq 0$, so
$$\deg(gr_{{\mathbb{F}}_0}(h_i))= \deg(x_{ij})+a_{0 j}.$$ Let $h_i=(y_{i 1},\cdots,y_{i \beta_0})$.
Since $h_i$ is a homogeneous element of $\ker(d_0)$ by (\ref{gr}), $x_{ij}=y_{i j}$ and then  we have
$$\deg(h_i)=\deg(y_{ij})+a^{'}_{0 j}=\deg(gr_{{\mathbb{F}}_0}(h_i))+c,$$ for some $c\in \Delta_{\mathbb{M}}(M). $
Hence $\Delta_{\mathbb{F}_{0}}(\Ker(d_{0})) \subseteq \Delta_{\mathbb{M}}(M)$
\vskip 2mm
\indent  We prove the result by inductive steps on  $n>0. $   Assume we  have defined filtered
graded free modules $F_{0},\cdots,F_{n-1}$ with special filtration $\mathbb{F}_{0},\cdots,\mathbb{F}_{n-1}$ such that
\[F_{n-1}\xrightarrow{d_{n-1}}F_{n-2}\rightarrow \cdots \rightarrow F_{0}\xrightarrow{d_{0}} M\rightarrow 0\]
is a part of a graded free resolution of $M$ and for every
$i<n   $ we assume:
\begin{enumerate}
\item[$(i)$] $gr_{\mathbb{F}_{i}}(F_{i})=G_{i},\ \ gr_{\mathbb{F}_{i}}(d_{i})=\delta_{i} $
\item[$(ii)$] $\Ker(d_{i})$ has a homogeneous standard basis with respect to $\mathbb{F}_i, $
\item [$(iii)$]$ \Ker({\delta_{i}})=gr_{\mathbb{F}_{i}}(\Ker(d_{i})) $
\item[$(iv)$] $\Delta_{\mathbb{F}_{i-1}}(\Ker(d_{i-1})) \subseteq \Delta_{\mathbb{M}}(M). $
\end{enumerate}
\vskip 2mm

 Let $\{g_{1},\cdots,g_{\beta_{n}}\}$  be a minimal homogeneous
 standard basis of $\Ker(d_{n-1}).  $  We know that
  $\Ker(\delta_{n-1})=gr_{\mathbb{F}_{n-1}}(\Ker(d_{n-1}))=\langle gr_{\mathbb{F}_{n-1}}(g_{1}),\cdots,gr_{\mathbb{F}_{n-1}}(g_{\beta_{n}})\rangle.$
  Set $a^{'}_{n i}:=\deg(g_{i}) \ \text{and} \ a_{n i}:=v_{\mathbb{F}_{n-1}}(g_{i}). $
 Define now the  graded free $R$-module
$$F_{n}=\bigoplus^{\beta_{n}}_{i=1}R(-a^{'}_{n i}) $$
equipped with the special filtration $\mathbb{F}_n$ defined by the integers $v_{\mathbb{F}_{n-1}}(g_{i})$ for $i=1, \dots, \beta_n,$ and let  $$ d_{n}:F_{n}\rightarrow F_{n-1}$$
defined by  $d_{n}(e_{n i })=g_{i}. $
Consider  the graded free $R$-module
 \[G_{n}=\bigoplus^{\beta_{n}}_{i=1}R(-a_{n i}) ,\quad\ \delta_{n}:G_{n}\rightarrow G_{n-1}\]
 such that  $\delta_{n}(e_{n i })=gr_{\mathbb{F}_{n-1}}(g_{i}).$
 Then
 \[F_{n}\xrightarrow{d_{n}} F_{n-1}\xrightarrow{d_{n-1}} F_{n-2} \]
 \[G_{n}\xrightarrow{\delta_{n}} G_{n-1}\xrightarrow{\delta_{n-1}} G_{n-2}\]\\
 are exact complexes and we have $gr_{\mathbb{F}_{n}}(d_{n})=\delta_{n}.$  Since $\{g_{1},\cdots,g_{\beta_{n}}\}$ is a minimal homogeneous standard basis of $\Ker(d_{n-1}), $
  again by Theorem \ref{stand-base},
we get $\Ker(\delta_{n})=gr_{\mathbb{F}_{n}}(\Ker(d_{n}))$
and $\Ker(d_{n})$ admits  a homogeneous standard basis with respect to $\mathbb{F}_n$.

 Let $gr_{\mathbb{F}_{n-1}}(g_{i})=^{t}(c_{i 1},\cdots,c_{i \beta_{n-1}})$, there exists $j$ such that
 $c_{i j}\neq 0$ and
 \[a_{n i}=\deg(c_{i j})+a_{n-1 j}\]
 Let $g_{i}=^{t}(u_{i 1},\cdots,u_{i \beta_{n-1}})$. Since  $g_{i}$
 is a homogeneous element of $\Ker(d_{n-1})$
  by (\ref{gr}),   $u_{i j}=c_{i j}$ and
\begin{equation} \label{shift-eq2}
a^{'}_{n i}=\deg(c_{i j})+a^{'}_{n-1 j}.
\end{equation}\\
By inductive assumption  there exists $c \in\Delta_{\mathbb{M}}(M) $ such that
$a^{'}_{n-1 j}=a_{n-1 j}+c$,
thus we get
\begin{equation}\label{shift-eq}
 a^{'}_{n i}=a_{n i}+c.
 \end{equation}
  and hence $\Delta_{\mathbb{F}_n}(\Ker(d_n)) \subseteq \Delta_{\mathbb{M}}(M)   $  as required. We can repeat  the  inductive process on $\Ker(d_{n}). $
\end{proof}
\begin{rem}\label{home-typ1}  We remark that if for some integers $i, j $  we have
$Tor_i^R(M,k)_j \neq 0,  $   then $Tor_i^R(gr_{\mathbb{M}}(M)),k)_{j-c} \neq 0 $ for some $c \in \Delta_{\mathbb{M}}(M).$

\noindent In fact denote by $\mathbf{F}^{min}$   a  minimal graded free resolution
of $M,  $ since $\mathbf{F}^{min}$  is a direct summand of  $\mathbf{F}$, we get
   \[\D(F_{i}^{min})\subseteq \D(F_{i})\quad \text{for all} \ i \geq 0.\]
For more details see  also \cite[3.1]{RSha}. Hence  the remark  follows from Proposition \ref{free-res} (iv).
Always  from Proposition \ref{free-res} it  will useful  to highlight   that  for all $i\geq 0$ and each  $b\in \D(G_i)$ there exist $a\in\D(F_i)$ and $c\in \Delta_{\mathbb{M}}(M) $ such that $a=b+c.$
 \end{rem}

\begin{thm}\label{shift-inq}
Let $R$ be a standard graded algebra
and let $M$ be a graded $R$-module equipped with the  filtration $\mathbb{M}$.
 Then for all $i\geq0$
\[t_{i}(M)\leq t_{i}(gr_{\mathbb{M}}(M))+v_{\mathbb{M}}(M).\]
Furthermore if $M$ is of homogeneous type with respect to $\mathbb{M}, $ then
\[t_{i}(gr_{\mathbb{M}}(M))+u_{\mathbb{M}}(M)\leq t_{i}(M)\leq t_{i}(gr_{\mathbb{M}}(M)))+v_{\mathbb{M}}(M).\]

s
\end{thm}
\begin{proof}
By Proposition \ref{free-res}, there exists a filtered graded free resolution $\FF$ of $M$ such that $gr_{\mathbb{F}}(\FF)=\GG$ is a  graded minimal free resolution of $gr_{\mathbb{M}}(M)$.  Let $i\geq0, $  we have $t_{i}(M)\in D(F_{i})$. Hence by  Proposition \ref{free-res} (iv) and Remark \ref{home-typ1},
 there exist  $b\in \D(G_{i})$ and
$c\in \Delta_{\mathbb{M}}(M)$ such that\\
\[t_{i}(M)=b+c. \]\\
Since  $b \leq  t_{i}(gr_{\mathbb{M}}(M))) $ and $c\leq v_{\mathbb{M}}(M), $ we get $t_{i}(M)\leq t_{i}(gr_{\mathbb{M}}(M)))+v_{\mathbb{M}}(M).$
If $M$ is of homogeneous type, then $\mathbf{F}$ is minimal. Now again  by Remark \ref{home-typ1}, for every  $i\geq0$ there exists
$a^{'}\in \D(F^{min}_{i})$ and $c^{'}\in\Delta_{\mathbb{M}}(M)$ such that\\
\[t_{i}(gr_{\mathbb{M}}(M))=a^{'}-c^{'}.\]\\
 Considering the fact that $u_{\mathbb{M}}(M)\leq c^{'}$,
we obtain
$$ t_{i}(gr_{\mathbb{M}}(M))+u_{\mathbb{M}}(M)\leq  t_{i}(gr_{\mathbb{M}}(M)) +c' = a' \leq t_{i}(M).$$

\end{proof}
\begin{cor}\label{reg-inq}
With the above notation and assumptions, we have
\[\reg_{R}(M)\leq \reg_{R}(gr_{\mathbb{M}}(M))+v_{\mathbb{M}}(M).\]
Moreover if $M$ is of homogeneous type with respect to $\mathbb{M}, $ then
\begin{equation}\label{reg-inq-hom}
\reg_{R}(gr_{\mathbb{M}}(M))+u_{\mathbb{M}}(M\leq \reg_{R}(M)\leq \reg_{R}(gr_{\mathbb{M}}(M))+v_{\mathbb{M}}(M).
\end{equation}
\end{cor}
\begin{proof}
It is a direct consequence of Proposition \ref{shift-inq}.
\end{proof}
\begin{cor} \label{corol} 
If
  $\reg_{R}(gr_{\mathbb{M}}(M))< \infty, $ then $\reg_{R}(M)< \infty$.
Furthermore the converse holds,  provided   $M$ is of homogeneous type with  respect to $\mathbb{M}$.
\end{cor}
The assumption of being of homogeneous type cannot be deleted as  the following example shows.
\begin{exam} \label{contro}
Let $k$ be a field  and $R=k[X,Y]/(X^3)$. Let $x,y$ be
 the residue class of $X,Y$ in $R$ and set $\m=(x,y)$.
 Consider the module $M$ whose minimal graded free resolution is
\[0\rightarrow R(-3)\xrightarrow{\binom{x^{2}}{y^{3}}}R(-1)\bigoplus R(0)\rightarrow 0\]
We have $M=R(-1)\bigoplus R(0)/N$ where $N= (x^{2},y^{3})R$
then $M^g\cong R\bigoplus R/N^{*}$
where $N^{*}$ is the submodule of $R\bigoplus R$
generated by the initial forms of elements of N with respect to the $\m$-adic filtration.
One can see that $N^{*}=\langle (x^{2},0),(0,xy^{3})\rangle$.
The minimal graded free resolution of $M^g$ is \\
\[\cdots\rightarrow R^{2}\xrightarrow{f_{i+1}}R^{2}\xrightarrow{f_{i}}
\cdots \xrightarrow{f_{2}}R^{2}\xrightarrow{f_{1}} R^{2}\rightarrow 0\]
where
\begin{equation*}
f_{1} = \left(
\begin{array}{ccc}
x^{2} & 0 \\
0 & xy^3
\end{array} \right)\qquad
f_{2k} = \left(
\begin{array}{ccc}
x & 0 \\
0 & x^{2}
\end{array} \right)\qquad
f_{2k+1} = \left(
\begin{array}{ccc}
x^{2} & 0 \\
0 & x
\end{array} \right)\quad for \ k\geq 1.
\end{equation*}

Therefore $M$ is not of homogeneous type and $\reg_{R}(M)<\infty, $
but $$ \reg_{R}(M^g)=~+\infty.$$
\end{exam}

\section{Linearity defect of a graded  module }

In this section we establish notation, provide the definitions  of linearity defect and of Koszul modules.
    We suggest
 \cite{CDR},    \cite{HI} and  \cite{S} for more details concerning  the   previous concepts. Even if this section  mainly concerns  with  graded standard $k$-algebras, we present the notions in the more general case of modules over local rings in view of   next section.

\vskip 2mm
\indent Let $(R,\m,k)$ be a local ring. A complex of $R$-modules
\[\mathbf{C}=\cdots \rightarrow C_{n+1}\rightarrow C_{n}\xrightarrow{\partial_{n}} C_{n-1}\rightarrow\cdots\]
is said to be minimal if $\partial_{n}(C_{n})\subseteq \m C_{n-1}$.
The {\em standard filtration}
$\mathfrak{F}$ of a minimal complex $\mathbf{C}$ is defined by subcomplexes $\{\mathfrak{F}_{i}\mathbf{C}\},$
where $(\mathfrak{F}^{i}\mathbf{C})_{n}=\m^{i-n}C_{n}$ for all $n\in\mathbb{Z}$.
with $\m^{j}=0$ for $j\leq 0$. The associated graded
complex with respect to this filtration is denoted by $\lin^{R}(\mathbf{C})$,
and called the {\em linear part of $\mathbf{C}$}.
By construction, $\lin^{R}(\mathbf{C})$ is a minimal complex of
graded modules over the graded ring $R^{g}$ and
it has the property that
$\lin^{R}(\mathbf{C})_{n}=C^{g}_{n}(-n).$

\begin{defin}
An $R$-module $M$ is said to be Koszul if $\lin^{R}(\FF )$ is acyclic,
 where $\FF $ is a  minimal free resolution of $M$.

\end{defin}

Herzog and Iyengar \cite[1.7]{HI} introduced an invariant
which  is  a measure of how far is $M$   from being  Koszul.

\begin{defin}\label{linear defect}
Let $M$ be an $R$-module and $\FF $ its minimal free resolution. Define
\[\ld_{R}(M):=\sup\{i\in\mathbb{Z} |  \HH_{i}(\lin^{R}(\FF) \neq 0\}\]
and we say that $\ld_{R}(M)$  is  the linearity defect of $M.$
\end{defin}

 By the uniqueness of a minimal free resolution  up to isomorphism of complexes, one has that $\ld_R(M) $ does not depend on $\FF, $ but only on the module $M.$   It follows by the definitions  that $M$ is a Koszul $R$-module if and only if $\ld_{R}(M)=0. $ Note that,  for any integer $d$, one has $\ld_{R}(M)\leq d$ if and only if the $dth$ syzygy module $\Omega^{d}(M)$ of $M$ is Koszul.

 Next result  provides a characterization of Koszul modules, see   \cite[1.5]{HI}.

 \begin{prop} \label{HI}
The following facts are equivalent:
 \begin{itemize}
 \item[(1)]  $M$ is Koszul,
 \item[(2)]  $\ld_R(M)=0,$
 \item[(3)]  $\lin(\FF) $ is a minimal free resolution of  $M^g $ as a $R^g$-module,
  \item[(4)] $M^g  $ has a linear resolution as a $R^g$-module.
 \end{itemize}
 \end{prop}

Note that if $M$ is Koszul, then $M^g$ has a linear resolution, hence $M$ is of homogeneous type,  that is  $\beta_j^R(M)= \beta_{j}^{R^g}(M^g)$ for every $j$.

\vskip 2mm Let R be a graded standard $k$-algebra.    If  $M $  is a finitely generated graded $R$-module, one
can define in the same manner the linearity defect  $\ld_R(M),$  by using a minimal
graded free resolution $\FF $ of $M$  over $R.$     Notice that  the complex  $\lin^R(\FF)$ is obtained from $\FF $ by replacing with $0$  all entries of degree $>1$  in the matrices representing the homomorphisms.
\noindent If a graded $R$-module $M$ has a linear resolution,
 then it is Koszul, but the converse
 fails in general; see \cite[1.9]{HI}.

  \vskip 2mm

Accordingly  with the classical definition given by Priddy,   $R$ is a Koszul graded $k$-algebra if $k$ has linear resolution as $R$-module. Because $k^g=k,  $  then by the previous proposition, $R$ is a Koszul graded $k$-algebra if and only if $k$ is a Koszul $R$-module.   If $R$ is a Koszul graded algebra,  it  was proved by Iyengar and  R\"omer  that a graded $R$-module $M$ is Koszul if and only if  $M$ is componentwise linear.


\vskip 2mm

Koszul modules have a special graded free resolution.   The following proposition  describes  the graded resolutions of Koszul modules extending to any graded standard algebra the result proved by Rossi and Sharifan in \cite[2.2.]{RS}.   \\

 \begin{prop} \label{extension} Let $R$ be a standard graded algebra  and let
$M$ be a finitely generated graded $R$-module  with $\D(M)=\{i_{1},\cdots,i_{s}\}$. Assume that
$M$ is Koszul.  Then for each $n\geq 1$ we have\\
\[\Tor^{R}_{n}(M,k)_{j}=0 \quad for\  j\neq i_{1}+n,\cdots,i_{s}+n.\]\\
Furthermore   if for some $n\geq 1$ and $1\leq r\leq s$, we have $\Tor^{R}_{n}(M,k)_{i_{r}+n}=0$,
then $\Tor^{R}_{m}(M,k)_{i_{r}+m}=0,$ for all $m\geq n.$
 \end{prop}
 \begin{proof}
 Since  $M$ is Koszul, then $M$ is of homogeneous type. We obtain the result as a consequence  of  Proposition \ref{free-res}. In fact  let $\FF $ and $\GG  $ be respectively  minimal graded
free resolutions  of $M$ and $M^g$   as $R$-modules. Observe that  $\Tor^{R}_{i}(M,k)_{j}=0$ if and only if $j\notin \D(F_{i}).$
By the assumption $\GG $ is linear,  then  by Remark \ref{home-typ1}  we get
\[\D(F_{n})\subseteq \{n+i_{1},\cdots,n+i_{s}\} \quad \text{for all} \ n\geq 1.\]
Let $n\geq 1$ and assume  $i_{r}+n+1\in \D(F_{n+1}). $ Again since $\GG $ is linear,
we conclude that there exists
$a\in \D(F_{n})$ such that $i_{r}+n+1=1+a. $
Hence  we get $i_{r}+n\in \D(F_{n}) $ and
 the conclusion follows inductively.
 \end{proof}

It is shown in \cite{HI} that  if $\ld_R(M)$ is finite, then $\reg_R(M)$ is finite as well.
   \\ We can  refine   this result giving a different proof and  a more precise information on  the regularity.

\begin{prop}\label{ld}
Let $R$ be a standard graded algebra  and let
$M$ be a finitely generated graded $R$-module.
If $\ld_{R}(M)=d<\infty$ then $$\reg_{R}(M)=\max\{t_{i}(M)-i: 0\leq i\leq d\}.$$
In particular, if $M$ is Koszul then $\reg_{R}(M)=t_{0}(M).$
\end{prop}
\begin{proof}
If $\ld_{R}(M)=0$,  that is $M$ is Koszul,  we have $\reg_{R}(gr_{m}(M))=0. $
Hence,  by Corollary \ref{reg-inq},  we conclude that  $\reg_{R}(M)\leq t_{0}(M)$.
 By  the fact that $t_{0}(M)\leq \reg_{R}(M), $ we get
 $\reg_{R}(M)=t_{0}(M).$ Now let $\ld_{R}(M)=d, $ then the
 $dth$ syzygy module $\Omega^{d}(M)$ is Koszul and by the above
 $\reg_{R}(\Omega^{d}(M))=t_{d}(M).$ This follows \\
 \[t_{i}(M)-i\leq t_{d}(M)-d \quad \text{for all }\ i\geq d  \]
 and then
 \[\reg_{R}(M)=\max\{t_{i}(M)-i: 0\leq i\leq d.\}\]

\end{proof}

\section{Linearity defect of a  module over a local ring  }
Let $(R,m)$ be a local ring and  let $M$ be a finitely generated $R$-module equipped
with the filtration $\mathbb{M}. $ Motivated by the questions raised  in the introduction our aim is to investigate the interplay between $\ld_R(M) $ and $\ld_{R^g} (gr_{\mathbb{M}}(M)).$

Recall that if $R$ is a regular local ring, Rossi and Sharifan in  \cite[3.6.]{RS}   proved that  $\ld_{R^g} (gr_{\mathbb{M}}(M))=0  $ implies $\ld_R(M)=0, $ provided $M$ is minimally generated by a standard basis with respect to $\mathbb{M}.$ Notice that Proposition \ref{extension} allows us to extend this result to any local ring by repeating verbatim the same proof in \cite{RS}. All the details can be found in PhD thesis of the second author, see \cite{AR}.

Observe that,  by Proposition \ref{HI},  $\ld_R(M)=0$ implies $\ld_{R^g}(M^g) =0, $ hence $M$ is Koszul if and only if $M^g$ is Koszul.
\vskip 2mm

Unfortunately this statement   cannot be extend to $\ld_R(M)=d. $ Example \ref{contro} shows that $\ld_R(M)< \infty $ does not imply $\ld_{R^g}(M^g) < \infty. $ In fact there exists  a module such that  $\ld_R(M)< \infty $ having finite homological dimension,  but $\ld_{R^g}(M^g) $ is infinite since the regularity is infinite,   see Proposition  \ref{ld}.
\vskip 2mm
Because  our aim would be to give an answer to Question 1 based on $k$  as $R$-module,   one can ask Question 2   for cyclic modules. Next example disproves the assertion.

\begin{exam} \label{cyclic}
Consider $R=k[[x,y,z,u]]/(x^3)$  and let $ J=(x^3, y^2+x^2, z^2y+u^4)/(x^3).$ The ideal $J$ is generated by a complete intersection, hence the homological dimension of $R/J$ as $R$-module is finite and so the linearity defect is finite.

\noindent Note that $R^{g}$  is a hypersurface, and  since $\dep(R^g)-\dep((R/J)^g)=3$ by \cite[Theorem 5.1.1]{Av} the minimal graded $R^{g}$-free
resolution of  $(R/J)^{g}$
becames periodic of period 2 and     $\beta_{i}((R/J)^{g}) =
\beta_{i+1}((R/J)^{g})$ for $i \geq 3.  $   By computing the  3th and 4th syzsgies of $(R/J)^g$ we deduce   that  $\reg_{R^g}((R/J)^{g})$ is  infinite, therefore $\ld((R/J)^{g})
$ is  infinite.
\end{exam}

We will  prove that,  if $M$ is of homogeneous type,  the finiteness of  $\ld_R(M)$ controls the regularity of $gr_{\mathbb{M}}(M).$  Notice that the finiteness of the regularity of a graded module does not imply the finiteness of the  linearity defect.
\vskip 2mm

We need a technical lemma. With the notation fixed in Proposition \ref{free-res}, let $(\GG, \delta.)$
be the minimal graded $R^{g}$-free resolution of $gr_{\mathbb{M}}(M)$.
We build up an  $R$-free resolution $(\FF, d.)$
of $M$ and denote by $\mathcal{M}_{n}=(m_{r s})$
the corresponding matrix associated to the differential map $d_{n}: F_n \to F_{n-1}$.

\begin{lem}\label{ld-hom}
With the above notations, if for all $n\geq 1$ the matrix
$\mathcal{M}_{n}$ has an entry of $\m$-adic valuation $\leq 1$ in  each column,
then $$\reg_{R^{g}}(gr_{\mathbb{M}}(M))= t_{0}(gr_{\mathbb{M}}(M)).$$
\end{lem}
\begin{proof}
Let \[\mathbf{G}_{.}=\cdots\rightarrow \bigoplus_{j=1}^{\beta_{1}}R^{g}(-a_{1 j})\xrightarrow{\delta_{1}}\bigoplus_{j=1}^{\beta_{0}}R^{g}(-a_{0 j})\xrightarrow{\delta_{0}} gr_{\mathbb{M}}(M)\rightarrow 0\]
be the minimal graded free resolution of $gr_{\mathbb{M}}(M)$ and
$U_{j}=(u_{r s})$ be the $j$-th degree matrix of $gr_{\mathbb{M}}(M)$,
where \[u_{r s}=a_{j s}-a_{j-1 r}.\]
We denote by $\mathcal{M}_{j}^{*}=(n_{r s})$ the corresponding
matrix associated to $\delta_{j}$. By  Proposition \ref{free-res} (ii) and (iii), the columns of
$\mathcal{M}_{j}^{*}$ are the initial forms of the corresponding
columns  of $\mathcal{M}_{j}$ with respect to the special filtration
which has been defined on $F_{j-1}. $  In particular the degree
matrix $U_{j}$ controls the valuations of entries of $\mathcal{M}_{j}$
with respect to the $\m$-adic filtration. We have

\begin{equation}\label{value}
\begin{cases} v_{R}(m_{r s}) = u_{r s} &\text{if}\quad n_{r s}\neq 0\\
v_{R}(m_{r s}) >u_{r s} & otherwise
\end{cases}
\end{equation}
\vskip 2mm

We prove by induction on $j\geq 0$  that for every $j$ and
$1\leq s\leq \beta_{j}$ there exists $i$, $1\leq i\leq \beta_{0}$
such that $a_{j s}\leq j+a_{0 i}$. The case $j=0$ is obvious. Let $j>0$ and suppose
that the result has been proved for every $0 \le i<j$.
Let $1\leq s \leq \beta_{j}$, by the assumption
there exists $r$ with $1\leq r \leq \beta_{j-1}$ such that for
the entry $m_{r s}$ of $\mathcal{M}_{j-1}$ we have
$v_{R}(m_{r s})\leq 1$. If  $v_{R}(m_{r s})=0$, then $n_{r s}=0$,
since $\mathbf{G}_{.}$ is minimal. Then from (\ref{value})
we get $a_{j s}<a_{j-1 r}. $ By inductive assumption  there exists
$i$,  $1\leq i\leq \beta_{0}$ such that $a_{j-1 r}\leq j-1+a_{0 i}$.
Thus we conclude  $a_{j s}\leq j+a_{0 i}$.
If $v_{R}(m_{r s})=1$,  by (\ref{value})  we may have two cases

\begin{equation*}
\begin{cases} a_{j s}=a_{n-1 r}+1\\
 a_{j-1 r}+1>a_{j s}
\end{cases}
\end{equation*}\\
By inductive assumption  in each case we get
\[a_{j s}\leq j+ a_{0 i}\quad \text{for some}\quad  1\leq i\leq \beta_{0}.\]
Therefore we conclude that for each $j>0$ there exists
$i ,1\leq i\leq \beta_{0}$ such that $t_{j}(gr_{\mathbb{M}}(M))\leq j+ a_{0 i}.$
This holds  that
 \[t_{j}(gr_{\mathbb{M}}(M))\leq j+ t_{0}(gr_{\mathbb{M}}(M))\]
and hence
 \[\reg_{R^{g}}(gr_{\mathbb{M}}(M))=t_{0}(gr_{\mathbb{M}}(M)).\]

\end{proof}

\begin{thm} \label{ldlocal}
Let $(R,m)$ be a local ring and let  $M$ be a Koszul module
equipped with the filtration $\mathbb{M}$. If $M$
is of homogeneous type with respect to $\mathbb{M}$,
then \[\reg_{R^{g}}(gr_{\mathbb{M}}(M))=t_{0}(gr_{\mathbb{M}}(M)).\]
\end{thm}
\begin{proof}
Let $\mathbf{G} $ be the minimal graded $R^{g}$-free
resolution of $gr_{\mathbb{M}}(M)$.
Since $M$ is of homogeneous type, then the resolution  $\mathbf{F} $ is minimal. By the assumption $\lin(\mathbf{F} )$ is acyclic,
therefore each column of the matrices associated to the
differential maps of $\lin(\mathbf{F} )$ has a non-zero linear form.
This follows that each column of the corresponding matrices of
$\mathbf{F} $ has an entry of valuation one with respect to the $\m$-adic
filtration. Now the conclusion follows from Lemma  \ref{ld-hom}.
\end{proof}
\begin{cor}\label{ld- local}
Let $(R,\m)$ be a local ring and  let  $M$ be a finitely generated $R$-module
equipped with the filtration $\mathbb{M}$. Let $\ld_{R}(M)=d<\infty$ and assume that $M$
is of homogeneous type with respect to $\mathbb{M}$, then
\[\reg_{R^{g}}(gr_{\mathbb{M}}(M))= \max\{ t_{i}(gr_{\mathbb{M}}(M))-i ; 0\leq i\leq d\}.\]
\end{cor}
\begin{proof}
Since $\ld_{R}(M)=d, $ then the dth syzygy $\Omega^d(M) $   of $M$ is Koszul. Hence we may apply Theorem  \ref{ldlocal} and we conclude that
$\reg_{R^{g}}(gr_{\mathbb{F}}(\Omega^d(M)))=t_{0}(gr_{\mathbb{F}}(\Omega^d(M)).$ We proceed now as in the proof of Proposition \ref{ld}.
\end{proof}

We are ready now to give   a  partial answer to Question 1.
\begin{prop} \label{k}
Let $(R,\m,k)$ be a local ring. If $\ld_{R}(k)<\infty$ and
$k$ is of homogeneous type,  then $\ld_{R}(k)=0.$
\end{prop}
\begin{proof}
By Corollary \ref{ld- local} we get $\reg_{R^{g}}(k)<\infty,  $ then according to
Avramov and Peeva \cite[ Theorem 2]{AP}  we obtain $\ld_{R}(k)=0.$
\end{proof}


\begin{thebibliography}{2}






\bibitem{AR} Ahangari Maleki .R, {\em Koszul rings and  modules and Castelnuovo-Mumford regularity}, PhD thesis, Kharazmi University (Iran), 2013.
\bibitem{Av} Avramov, L.L, {\em Infinite free resolutions}, Six lectures on commutative algebra
(Bellaterra, 1996), Progr. Math. 166, Birkhauser, Basel,  1--118 (1998) .
\bibitem{A-E} Avramov, L.L., Eisenbud, D.   {\em Regularity of modules over a Koszul algebra}, J.Algebra  153: 85--90 (1992).

\bibitem{AP}  Avramov, L.L., Peeva, I {\em Finite regularity and Koszul algebras}, Amer. J. Math. 123, (2001)  275–281.
\bibitem{CDR}  Conca A.,  De Negri E.,  Rossi  M.E., Koszul
algebras and regularity,   Irena Peeva Ed.. Commutative Algebra: Expository papers dedicated to David Eisenbud,  Springer, 285- 315 (2013).

\bibitem{HI}Herzog, J., Iyengar, S.  {\em Koszul modules}, J. Pure Appl. Algebra 201, (2005) 154--188.

\bibitem{HRV} Herzog, J.,  Rossi, M.E,  Valla, G. {\em On the depth of the symmetric algebra}, Trans. Amer. Math. Soc. 296(2),  (1986)  577-–606.


\bibitem{IR}Iyengar, S., R\"{o}mer, T.  {\em Linearity defects of modules over commutative rings}, J.Algebra. 322 (2009), 3212 --3237.

\bibitem{RV}  Robbiano, L.,   Valla, G. {\em Free resolutions for special tangent cones}, Commutative algebra
(Trento, 1981),  Lecture Notes in Pure and Appl. Math. 84, Dekker, New York, 253–-274
(1983).



\bibitem{RS} Rossi, M. E,  Sharifan, L. {\em Minimal free resolution of a finitely generated module over a regular local ring}, J. Algebra 322 , no. 10,  (2009) 3693–-3712 .

\bibitem{RSha}Rossi, M. E,  Sharifan, L. {\em Consecutive cancellation in Betti numbers of local
rings}, Proc. Amer. Math. Soc, 138, (2010)  61--73.



\bibitem{RV-Lect} Rossi, M. E, Valla, G. {\em Hilbert functions of filtered modules},
Lecture Notes of the Unione Matematica Italiana, {\bf9}. Springer-Verlag, Berlin; UMI, Bologna, 2010.

\bibitem{S-lin}\c{S}ega, L.M. {\em On the linearity defect of the residue field}, arXiv:1303.4680v1 [math.AC]  (2013).

\bibitem{S}\c{S}ega, L.M.   {\em Homological properties of powers of the maximal ideal of a local ring},
J. Algebra 241 (2001),  827--858.

\bibitem{Shibuta} Shibuta, T. {\em Cohen–-Macaulyness of almost complete intersection tangent cones}, J. Algebra 319 (8) (2008), 3222–-3243.

\end{thebibliography}
\end{document}